\documentclass[a4paper,12pt,reqno]{amsart}

\usepackage{amsmath}
\usepackage{amssymb}
\usepackage{amsfonts}
\usepackage{graphicx}
\usepackage[colorlinks]{hyperref}
\renewcommand\eqref[1]{(\ref{#1})} 

\graphicspath{ {images/} }
\title[Global existence and blow-up of solutions]{Global existence and blow-up of solutions to porous medium equation for Baouendi-Grushin operator}

\author[A. Dukenbayeva]{Aishabibi Dukenbayeva}
\address{
  Aishabibi Dukenbayeva:
    \endgraf
  Suleyman Demirel University
  \endgraf
  Kaskelen, Kazakhstan
  \endgraf
  and
  \endgraf
  Institute of Mathematics and Mathematical Modeling
  \endgraf
 Almaty, Kazakhstan
  \endgraf
 {\it E-mail address} {\rm aishabibi.dukenbayeva@gmail.com}
  }


\thanks{This research is funded by the Committee of Science of the Ministry of Science and Higher Education of the Republic of Kazakhstan (Grant No. AP14972714).}

\newtheoremstyle{theorem}
{10pt}          
{10pt}  
{\sl}  
{\parindent}     
{\bf}  
{. }    
{ }    
{}     
\theoremstyle{theorem}

\numberwithin{equation}{section}
\theoremstyle{plain}
\newtheorem{thm}{Theorem}[section]

\newtheorem{lem}[thm]{Lemma}

\theoremstyle{definition}

\newtheorem{rem}[thm]{Remark}

\newtheoremstyle{defi}
{10pt}          
{10pt}  
{\rm}  
{\parindent}     
{\bf}  
{. }    
{ }    
{}     
\theoremstyle{defi}

\begin{document}
	\begin{abstract}
		In this note, we show a global existence and blow-up of the positive solutions to the initial-boundary value problem of the nonlinear porous medium equation related to Baouendi-Grushin operator. Our approach is based on the concavity argument and the Poincar\'e inequality for Baouendi-Grushin vector fields from \cite{SY20}, inspired by the recent works \cite{RST23_strat} and \cite{ST-21}.
	\end{abstract}
	\maketitle
	\section{Introduction}
	Let $z:=(x,y):=(x_{1},...,x_{m}, y_{1},...,y_{k})\in \mathbb{R}^{m}\times\mathbb{R}^{k}$ with $m,k\geq1$ and $m+k=n$. Consider the vector fields
$$X_{i}=\frac{\partial}{\partial x_{i}}, \;i=1,...,m, \;\;\; Y_{j}=|x|^{\gamma}\frac{\partial}{\partial y_{j}},\;\gamma\geq0,\;j=1,...,k.$$
Then the corresponding sub-elliptic gradient and Baouendi-Grushin operator on $\mathbb{R}^{m+k}$ are defined by
\begin{equation}\label{subgrad}
\nabla_{\gamma}:=(X_{1},...,X_{m}, Y_{1},...,Y_{k})=(\nabla_{x}, |x|^{\gamma}\nabla_{y}).
\end{equation}
and
\begin{equation}\label{Grush_op}
\Delta_{\gamma}=\sum_{i=1}^{m}X_{i}^{2}+\sum_{j=1}^{k}Y_{j}^{2}=\Delta_{x}+|x|^{2\gamma}\Delta_{y}=\nabla_{\gamma}\cdot \nabla_{\gamma},
\end{equation}
respectively, where $\Delta_{x}$ and $\Delta_{y}$ stand for the standard Laplacians in the variables $x\in \mathbb{R}^{m}$ and $y\in \mathbb{R}^{k}$.

Let $D \subset \mathbb{R}^{m+k}$ be a bounded domain (open and connected) supporting the divergence formula and $D \backslash \{(x, y) \in \overline{D}: x=0\}$ consists of only one connected component. In this note, we study the global existence and blow-up of the positive solutions to the initial-boundary problem of the nonlinear porous medium equation related to the Baouendi-Grushin operator
	\begin{align}\label{main_eqn_p>2}
	\begin{cases}
	u_t -\Delta_{\gamma} (u^\ell) = f(u), \,\,\, & (x,y) \in D,\,\, t>0, \\ 
	u(x,y,t)  =0,  \,\,\,& (x,y)\in \partial D,\,\, t>0, \\
	u(x,y,0)  = u_0(x,y)\geq 0,\,\,\, & (x,y) \in \overline{D}, 
	\end{cases}
	\end{align}
	where $\ell \geq 1$, $f$ is locally Lipschitz continuous on $\mathbb{R}$, $f(0)=0$, and such that $f(u)>0$ for $u>0$. Furthermore, we suppose that $u_0$ is a non-negative and non-trivial function in $C^1(\overline{D})$ with $u_0(x,y)=0$ on the boundary $\partial D$. 
	
Recall that for an even positive integer $\gamma$ the Baouendi-Grushin operator can be represented by a sum of squares of smooth vector fields satisfying the H\"{o}rmander rank condition
$${\rm rank} \;{\rm Lie} [X_{1},...,X_{m}, Y_{1},...,Y_{k}]=n.$$
The anisotropic dilation attached to $\Delta_{\gamma}$ on $\mathbb{R}^{m+k}$ and the homogeneous dimension are defined by
$$\delta_{\lambda}(z)=(\lambda x, \lambda^{1+\gamma} y), \quad \lambda>0$$
and
\begin{equation}\label{hom_dim}
Q=m+(1+\gamma)k,
\end{equation}
respectively. Recall also that a change of variables formula for the Lebesgue measure gives that
$$d\circ\delta_{\lambda}(x,y)=\lambda^{Q}dxdy.$$

We will also use the Sobolev space $H_{0}^{1, \gamma}(D)$ obtained as completion of $C_{0}^{\infty}(D)$ with respect to the norm
$$
\|f\|_{H_{0}^{1, \gamma}(D)}=\left(\int_{D}\left|\nabla_{\gamma} f\right|^{2} d x d y\right)^{\frac{1}{2}}.
$$

One notable example of nonlinear parabolic equations is the porous medium equation. It describes a variety of phenomena, encompassing fluid flow, heat transfer, and diffusion processes. Furthermore, its applications span across various fields, including mathematical biology, lubrication, boundary layer theory, and etc. In the case $\gamma=0$, hence $\Delta_{\gamma}$ reduces to the usual Laplacian $\Delta$, for existence and nonexistence results for the problem \eqref{main_eqn_p>2} we refer to e.g.  \cite{Ball, Band-Brun, Chen-Fila-Guo, Ding-Hu, Deng-Levine, Gal-Vaz-97, Gr-Mu-Po-13, Hayakawa, Ia-San-14, Ia-San-19, Levine90, LP1, ST-21,  Sam-Gal-Ku-Mik, Souplet} and to \cite{Gr-Mu-Pu-1, Gr-Mu-Pu-2} for the fractional porous medium equation as well as to \cite{Gr-Mu-Pu-3} on Cartan-Hadamard manifolds. Applying the concavity method, for blow-up and global existence of the solution to the problem \eqref{main_eqn_p>2} Schaefer \cite{Sch09} established a condition on the initial data when $\gamma=0$ and $f(u)=Cu^{p}$ for some positive constant $C$ and $p>\ell\geq 1$ on a smoothly bounded domain in $\mathbb{R}^{n}$. Recently, in \cite{ST-21} the authors studied similar questions for the double nonlinear porous medium equation. For a more comprehensive mathematical study of the porous medium equation we refer to Vazquez's book \cite{Vaz}.

In this note, we are interested in extensions of such results from the Laplacian to the Baouendi-Grushin operator. In the setting of hypoelliptic operators, the authors in \cite{PohVer} investigated blow-up of the solutions to the following semi-linear diffusion equation on the Heisenberg group $\mathbb{H}^{n}$
	$$
	\frac{\partial u(x,t)}{\partial t}-\mathcal{L} u(x,t)=|u(x,t)|^{p},\,\,\,\,\,(x,t)\in \mathbb{H}^{n} \times(0,+\infty),
	$$
 where $\mathcal{L}$ is the sub-Laplacian on $\mathbb{H}^{n}$. In this setting, we also refer to \cite{AAK1, AAK2, DL1, JKS1, JKS2} for blow-up type results for the semi-linear diffusion and pseudo-parabolic equations. Recently, in \cite{RST23_strat} the authors obtained global existence and blow-up type results for the problem \eqref{main_eqn_p>2} but with a sub-Laplacian on stratified Lie groups. We also refer to \cite{RY} where the authors found the Fujita exponent on general unimodular Lie groups. 

 One of the main difference of the Baouendi-Grushin operator than above sub-Laplacians on stratified or unimodular Lie groups is the Baouendi-Grushin operator cannot be represented by a sum of squares of smooth vector fields satisfying the H\"{o}rmander condition when $\gamma$ is not even integer. Moreover, unlike the setting of the Heisenberg group, the Baouendi-Grushin operator can be reduced to the usual Euclidean Laplacian when $\gamma=0$ thus implying the corresponding results for the classical Laplacian. 

 Before stating our main results, let us recall the following Poincar\'e inequality for Baouendi-Grushin vector fields from \cite{SY20}: 
		\begin{lem}\label{lem1}
			 			 Let $D \subset \mathbb{R}^{m+k}$ be a bounded domain (open and connected) supporting the divergence formula and $D \backslash \{(x, y) \in \overline{D}: x=0\}$ consists of only one connected component. For every function $u \in H_{0}^{1, \gamma}(D)$ we have 
			\begin{equation}
			\int_{D} |\nabla_{\gamma} u|^2 dxdy \geq \lambda_{1} \int_{D} |u|^2 dxdy,
			\end{equation}
			where $\lambda_{1}$ is the first eigenvalue of the Dirichlet Baouendi-Grushin operator on $D$.
				\end{lem}
				\begin{rem}
				    For more details on spectral properties of the Dirichlet Baouendi-Grushin operator we refer to e.g. \cite{AHKP08},  \cite{MSV15} and \cite{MP09}.
				\end{rem}
    	
Further, we always assume that $D$ is as in Lemma \ref{lem1}.

   Thus, our first result on the blow-up property is as follows:
\begin{thm}\label{thm_p>2}
				Assume that 
		\begin{equation}\label{condt_p}
		\alpha F(u) \leq u^{\ell} f(u) + \beta u^{2\ell} +\alpha\theta,\,\,\, u>0,
		\end{equation}
		where 
		\begin{equation*}
		F(u)=\frac{2\ell}{\ell+1}\int_{0}^{u}s^{\ell-1}f(s)ds, \,\,\, \ell\geq 1,
		\end{equation*}			
		for some
		\begin{equation*}
		\theta >0,\,\, 0<\beta\leq \lambda_{1}\frac{(\alpha - \ell -1)}{\ell+1}\,\,\, \text{ and } \,\, \alpha>\ell+1.
		\end{equation*} 
				Let the initial data $u_0 \in L^{\infty}(D)\cap H_{0}^{1, \gamma}(D)$ satisfy the inequality 	 \begin{equation}\label{J(1)}
		J_0:=   - \frac{1}{\ell+1} \int_{D} |\nabla_{\gamma} u^\ell_0(x,y)|^2 dxdy +  \int_{D} (F(u_0(x,y))-\theta) dxdy >0.
		\end{equation} 
		Then any positive solution $u$ of the problem \eqref{main_eqn_p>2} blows up in finite time $T^*,$ that is, there exists 
		\begin{equation}\label{T}
		0<T^*\leq \frac{M}{\sigma \int_{D}u_0^{\ell+1}(x,y)dxdy},
		\end{equation}
		such that 
		\begin{equation}
		\lim_{t\rightarrow T^*} \int_{0}^t \int_{D} u^{\ell+1} (x,y,\tau) dxdy d\tau = +\infty,
		\end{equation}
		where $M>0$ and $\sigma = \frac{\sqrt{2\ell\alpha}}{\ell+1}-1>0$.
		In fact, we can take 
		\begin{equation*}
		M = \frac{(1+\sigma)(1+1/\sigma)(\int_D u_0^{\ell+1}(x,y)dxdy)^2}{\alpha (\ell+1)J_0}.
		\end{equation*}
	\end{thm}
	\begin{rem} Recall that the condition on $f(u)$ was introduced by Chung and Choi \cite{Chung-Choi} for a parabolic equation with the classical Laplacian. We also refer to \cite{PhP-06} and \cite{Band-Brun} for special cases of this condition. For recent papers on such conditions, we can refer to \cite{RST23_strat}, \cite{ST-21}  and \cite{Bolys_wave}.
	\end{rem}
	\begin{rem}
	In particular, Theorem \ref{thm_p>2} for $\ell=1$ was obtained in \cite{SY20}. 
	\end{rem}
 The next result shows that under some assumptions, if a positive solution to \eqref{main_eqn_p>2} exists, its norm is globally controlled.  
	\begin{thm}\label{thm_GEp}
				Assume that
		\begin{equation}\label{global_cond-p}
		\alpha F(u) \geq u^{\ell} f(u) + \beta u^{2\ell} +\alpha\theta, \,\,\, u>0,
		\end{equation}
		where
		\begin{equation*}
		F(u)=\frac{2\ell}{\ell+1}\int_{0}^{u}s^{\ell-1}f(s)ds, \,\,\, \ell\geq 1,
		\end{equation*}
		for some
		\begin{equation*}
		\theta \geq 0, \,\, \alpha \leq 0 \,\,\, \text{ and } \,\,\, \beta \geq \lambda_{1}\frac{( \alpha-\ell-1 )}{\ell+1}.
		\end{equation*}
		
		Assume also that the initial data $u_0 \in L^{\infty}(D)\cap H_{0}^{1, \gamma}(D)$ satisfies the inequality 
		\begin{equation}\label{J(0)}
		J_0:= \int_{D} (F(u_0(x,y))-\theta) dxdy - \frac{1}{\ell+1} \int_{D} |\nabla_{\gamma} u^\ell_0(x,y)|^2 dxdy>0.
		\end{equation}
		If $u$ is a positive local solution of the problem \eqref{main_eqn_p>2}, then it is global with the property
		\begin{equation*}
		\int_D u^{\ell+1}(x,y,t) dxdy \leq  \int_D u^{\ell+1}_0(x,y)dxdy.
		\end{equation*}
		
	\end{thm}

	\section{Proofs}\label{sec1}

	\begin{proof}[Proof of Theorem \ref{thm_p>2}]
		Assume that $u(x,y,t)$ is a positive solution to \eqref{main_eqn_p>2}. Let us define
				\begin{equation*}
		E(t): = \int_{0}^t \int_{D} u^{\ell+1}(x,y, \tau) dxdy d\tau + M, \,\, t\geq 0,
		\end{equation*}
		with $M>0$ to be chosen later. Then, for $E(t)$ one can observe that
		\begin{multline*}
		E'(t)=\int_{D} u^{\ell+1}(x,y,t) dxdy \\= (\ell+1)\int_{D} \int_{0}^t u^{\ell}(x,y,\tau) u_{\tau}(x,y,\tau) d\tau dxdy + \int_{D} u^{\ell+1}_0(x,y) dxdy.
		\end{multline*}
Applying the H\"older and Cauchy-Schwarz inequalities, we note that
		\begin{align*}
		(E'(t))^2&\leq  (1+\delta)\left( \int_{D} \int_{0}^t (u^{\ell+1}(x,y,\tau))_{\tau} d\tau dxdy  \right)^2 + \left( 1+ \frac{1}{\delta}\right)\left( \int_{D} u_0^{\ell+1}(x,y)dxdy \right)^2 \\
		& = (\ell+1)^2(1+\delta)\left( \int_{D} \int_{0}^t u^{\ell}(x,y,\tau)  u_{\tau}(x,y,\tau)dxdy d\tau \right)^2   \\&+ \left( 1+ \frac{1}{\delta}\right)\left( \int_{D} u_0^{\ell+1}(x,y)dxdy \right)^2
		\end{align*}
		\begin{align}\label{eq-E2}
&= (\ell+1)^2(1+\delta)\left( \int_{D} \int_{0}^t u^{(\ell+1)/2 + (\ell-1)/2}(x,y,\tau)  u_{\tau}(x,y,\tau)dxdy d\tau \right)^2   \nonumber\\
		&+ \left( 1+ \frac{1}{\delta}\right)\left( \int_{D} u_0^{\ell+1}(x,y)dxdy \right)^2
		\nonumber\\&\leq (\ell+1)^2(1+\delta)\left( \int_{D} \left(\int_{0}^t u^{\ell+1} d\tau\right)^{1/2}\left( \int_{0}^t u^{\ell-1} u_{\tau}^2(x,y,\tau)d\tau\right)^{1/2} dxdy \right)^2  \nonumber \\
		&+ \left( 1+ \frac{1}{\delta}\right)\left( \int_{D} u_0^{\ell+1}(x,y)dxdy \right)^2 \nonumber \\
		& \leq (\ell+1)^2(1+\delta) \left(\int_{0}^t \int_{D} u^{\ell+1} dxdy d\tau\right)\left( \int_{0}^t \int_{D} u^{\ell-1} u_{\tau}^2(x,y,\tau)dxdyd\tau \right)   \nonumber\\
		&+ \left( 1+ \frac{1}{\delta}\right)\left( \int_{D} u_0^{\ell+1}(x,y)dxdy \right)^2
		\end{align}
		for any $\delta>0$. 
  
		For the second derivative of $E(t)$, by \eqref{condt_p}, Lemma \ref{lem1} and $0<\beta\leq \lambda_{1}$ we have
		\begin{align*}
		E''(t) &=(\ell+1) \int_{D} u^{\ell}(x,y,t) u_t(x,y,t) dxdy\\
		& = -(\ell+1) \int_{D} |\nabla_{\gamma} u^\ell(x,y,t)|^2 dxdy + (\ell+1)\int_{D} u^{\ell}(x,y,t)f(u(x,t))dxdy
		\\
		& \geq (\ell+1) \int_{D} \left( \alpha F(u(x,y,t)) -\beta u^{2\ell} (x,y,t) -\alpha \theta \right) dxdy\\& - (\ell+1) \int_{D} |\nabla_{\gamma} u^\ell(x,y,t)|^2 dxdy\\
		& =\alpha (\ell+1) \left( -\frac{1}{\ell+1} \int_{D} |\nabla_{\gamma} u^\ell(x,y,t)|^2 dxdy + \int_{D} (F(u(x,y,t))-\theta) dxdy \right)  \\
		&+ (\alpha-\ell-1) \int_{D} |\nabla_{\gamma} u^\ell(x,y,t)|^2 dxdy -\beta (\ell+1) \int_{D} u^{2\ell}(x, y,t) dxdy
		\end{align*}
\begin{align*}
		& \geq \alpha(\ell+1) \left( -\frac{1}{\ell+1} \int_{D} |\nabla_{\gamma} u^\ell(x,y,t)|^2 dxdy + \int_{D} (F(u(x,y,t))-\theta) dxdy \right)  \\
		& + \left( \lambda_{1}(\alpha - \ell-1) - \beta (\ell+1) \right)\int_{D} u^{2\ell}(x,y,t) dxdy\\
		&  \geq \alpha (\ell+1)\left( -\frac{1}{\ell+1} \int_{D} |\nabla_{\gamma} u^\ell(x,y,t)|^2 dxdy + \int_{D} (F(u(x,y,t))-\theta) dxdy \right)\\
		& =: \alpha (\ell+1)J(t)
  \end{align*}
  $$ = \alpha (\ell+1) J(0) + 2\alpha \ell \int_{0}^t \int_{D} u^{\ell-1}(x,y,\tau) u_{\tau}^2(x,y,\tau) dxdy d\tau,$$
  where we have used the following relation in the last line 
		\begin{multline}\label{Fp}
		J(t) = J(0) + \int_{0}^t \frac{d J(\tau)}{d\tau}d\tau=
  J(0) \\- \frac{1}{\ell+1} \int_{0}^t\int_{D} \frac{d}{d\tau}|\nabla_{\gamma} u^\ell(x,y,\tau)|^2 dxdy d\tau  + \int_{0}^t \int_{D} \frac{d}{d\tau} (F(u(x,y,\tau))-\theta) dxdy d\tau\\=
   J(0) - \frac{2}{\ell+1} \int_{0}^t\int_{D}  \nabla_{\gamma} u^\ell \cdot \nabla_{\gamma} (u^\ell(x,y,\tau))_{\tau} dxdy d\tau\\
		 +  \int_{0}^t \int_{D} F_u(u(x,y,\tau)) u_{\tau}(x,y,\tau) dxdy d\tau\\
  =J(0)+\frac{2}{\ell+1} \int_{0}^t \int_{D} (\Delta_{\gamma} (u^\ell )+ f(u))(u^\ell(x,y,\tau))_{\tau} dxdy d\tau \\=
  J(0)+\frac{2\ell}{\ell+1}\int_{0}^t \int_{D} u^{\ell-1}(x,y,\tau) u_{\tau}^2(x,y,\tau) dxdy d\tau.
		\end{multline}
Note that for $J_{0}$ from \eqref{J(1)} we have $J(0)=J_{0}$. Since $\alpha > \ell+1$ implies $\sigma=\delta= \frac{\sqrt{2\ell\alpha}}{\ell+1}-1>0$ and combining the above estimates, we arrive at
		\begin{align*}
		&E''(t) E (t) - (1+\sigma) (E'(t))^2\\
		&\geq  \alpha M(\ell+1) \left( -\frac{1}{\ell+1} \int_{D} |\nabla_{\gamma} u^\ell_0|^2 dxdy + \int_{D} (F(u_0)-\theta)dxdy\right) 
		\\&+ 2\ell\alpha \left( \int_{0}^t\int_{D} u^{\ell+1}(x, y,\tau) dxdy d\tau \right) \left(\int_{0}^t \int_{D} u_{\tau}^2(x,\tau) u^{\ell-1}(x,y,\tau) dxdy d\tau  \right)\\
		&- (\ell+1)^2(1+\sigma) (1+\delta) \left(\int_{0}^t \int_{D} u^{\ell+1} dxdy d\tau\right)\left( \int_{0}^t \int_{D} u^{\ell-1} u_{\tau}^2(x,y,\tau)dxdyd\tau \right)\\
		& - (1+\sigma)\left( 1+ \frac{1}{\delta}\right)\left( \int_{D} u_0^{\ell+1}(x,y)dxdy \right)^2 \\
		&\geq  \alpha M(\ell+1) J(0) - (1+\sigma)\left( 1+ \frac{1}{\delta}\right)\left( \int_{D} u_0^{\ell+1}(x,y)dxdy \right)^2.
		\end{align*}

		Recall that by \eqref{J(1)} we have $J(0)>0$, thus we can choose 
		$$M =\frac{(1+\sigma)\left( 1+ \frac{1}{\delta}\right)\left( \int_{D} u_0^{\ell+1}(x,y)dxdy \right)^2}{\alpha (\ell+1) J(0)}$$ 
		so that
		\begin{equation}
		E''(t) E (t) - (1+\sigma) (E'(t))^2 \geq 0
		\end{equation}
  holds. On the other hand, it means for $t\geq 0$ that 
		\begin{equation*}
		\frac{d}{dt} \left( \frac{E'(t)}{E^{\sigma+1}(t)} \right) \geq 0  \Rightarrow 	\begin{cases}
		E'(t) \geq \left( \frac{E'(0)}{E^{\sigma+1}(0)} \right) E^{1+\sigma}(t),\\
		E(0)=M.
		\end{cases}
		\end{equation*}
		Taking into account $\sigma = \frac{\sqrt{2\ell\alpha}}{\ell+1}-1>0$, we obtain
		\begin{align*}
		- \frac{1}{\sigma} \left( E^{-\sigma}(t) - E^{-\sigma}(0)  \right) \geq \frac{E'(0)}{E^{\sigma+1}(0)} t,
		\end{align*}
  which implies together with $E(0)=M$ that
		\begin{equation*}
		E(t) \geq \left( \frac{1}{M^{\sigma}}-\frac{ \sigma \int_{D} u^{\ell+1}_0(x,y)dxdy }{M^{\sigma+1}} t\right)^{-\frac{1}{\sigma}}.
		\end{equation*}
		Thus, we have observed that the blow-up time $T^*$ satisfies 
		\begin{equation*}
		0<T^*\leq \frac{M}{\sigma \int_{D} u_0^{\ell+1}dxdy},
		\end{equation*}
		as desired.
	\end{proof}	
Let us now prove Theorem \ref{thm_GEp}.
	\begin{proof}[Proof of Theorem \ref{thm_GEp}] Here, we work with the functional
			
		\begin{equation*}
		\mathcal E(t) = \int_{D} u^{\ell+1}(x,y,t)dxdy.
		\end{equation*}
		By using \eqref{global_cond-p}, Lemma \ref{lem1} and $\beta \geq \lambda_{1}\frac{( \alpha-\ell-1 )}{\ell+1}$, we have 
		\begin{align*}
		&\mathcal	E'(t) =(\ell+1) \int_{D} u^{\ell}(x,y,t) u_t(x,t) dxdy\\
		& = (\ell+1) \left(\int_{D} u^{\ell}(x,y,t) \nabla_{\gamma} \cdot (\nabla_{\gamma} u^\ell(x,y,t)) + \int_{D} u^{\ell}(x,y,t)f(u(x,y, t))dxdy\right) \\
		& = (\ell+1) \left(-\int_{D} |\nabla_{\gamma} u^\ell(x,y,t)|^2 dxdy + \int_{D} u^{\ell}(x,y,t)f(u(x,y,t))dxdy\right)
		\\
		& \leq (\ell+1) \left(-\int_{D} |\nabla_{\gamma} u^\ell(x,y,t)|^2 dxdy +  \int_{D} \left( \alpha F(u(x,y,t)) -\beta u^{2\ell} (x,y,t) -\alpha \theta \right) dxdy\right)\\
		& =\alpha (\ell+1) \left( -\frac{1}{\ell+1} \int_{D} |\nabla_{\gamma} u^\ell(x,y,t)|^2 dxdy + \int_{D} (F(u(x,y,t))-\theta) dxdy \right)  \\
		&- (\ell+1-\alpha) \int_{D} |\nabla_{\gamma} u^\ell(x,y,t)|^2 dxdy -\beta (\ell+1) \int_{D} u^{2\ell}(x,y, t) dxdy\\
		& \leq \alpha (\ell+1)\left( -\frac{1}{\ell+1} \int_{D} |\nabla_{\gamma} u^\ell(x,y,t)|^2 dxdy + \int_{D} (F(u(x,y,t))-\theta) dxdy \right)  \\
		& - \left( \lambda_{1}(\ell+1-\alpha ) + \beta (\ell+1) \right)\int_{D} u^{2\ell}(x, y,t) dxdy\\
		&  \leq \alpha (\ell+1)\left( -\frac{1}{\ell+1} \int_{D} |\nabla_{\gamma} u^\ell(x,y,t)|^2 dxdy + \int_{D} (F(u(x,y,t))-\theta) dxdy \right) \\
		&=\alpha (\ell+1)J(t),
		\end{align*}
  where we have used the same functional $J(t)$ as in the proof of Theorem \ref{thm_p>2}. \eqref{Fp}
  Note that taking into account \eqref{Fp} and $\alpha \leq 0$ in the estimate for $\mathcal E'(t)$ above one can get
		\begin{align}
		\mathcal E'(t) \leq \alpha (\ell+1) J(0) + 2\alpha \ell \int_{0}^t \int_{D} u^{\ell-1}(x,y,\tau) u_{\tau}^2(x,y,\tau) dxdy d\tau \leq 0,
		\end{align}
		which yields
		\begin{equation*}
		\mathcal E(t) \leq \mathcal E(0),
		\end{equation*}
		as desired.
	\end{proof}

\end{document}